\documentclass[a4paper]{amsart}

\usepackage[latin1]{inputenc}
\usepackage[T1]{fontenc}
\usepackage[english]{babel}
\usepackage{amsmath}
\usepackage{verbatim}
\usepackage{amsfonts}
\usepackage{mathtools}
\usepackage[all]{xy}
\usepackage{amssymb}
\usepackage{mathrsfs}
\usepackage{color}
\usepackage{enumerate}
\usepackage{enumitem}
\usepackage{multirow}
\usepackage{url}
\usepackage{accents}
\usepackage{graphicx}
\usepackage[hypertexnames=false,backref=page,pdftex,
 	pdfpagemode=UseNone,
 	breaklinks=true,
 	extension=pdf,
 	colorlinks=true,
 	linkcolor=blue,
 	citecolor=red,
 	urlcolor=blue,
 ]{hyperref}

\newcommand{\incl}[1][r]{\ar@<-0.2pc>@{^(-}[#1] \ar@<+0.2pc>@{-}[#1]}

\newcommand{\rk}{\operatorname{rk}}
\newcommand{\ad}{\operatorname{ad}}
\newcommand{\codim}{\operatorname{codim}}
\renewcommand{\mod}{{\operatorname{mod}}}
\renewcommand{\Im}{\operatorname{Im}}

\newcommand{\GL}{\rm GL}

\newcommand{\muz}{\mu^{-1}(0)}

\newcommand{\muzx}[1]{(\mu^{-1}(0))_{#1}}
\newcommand{\SPR}{/\!\!/\!\!/}
\newcommand{\GIT}{/\!\!/}

\renewcommand{\ll}{\mathfrak{l}} 
\newcommand{\lf}{\mathfrak{l}} 
\renewcommand{\gg}{\mathfrak{g}}
\newcommand{\af}{\mathfrak{a}}
\newcommand{\kk}{\mathfrak{k}}
\newcommand{\qq}{\mathfrak{q}}
\newcommand{\tf}{\mathfrak{t}}
\newcommand{\zz}{\mathfrak{z}}
\newcommand{\wfr}{\mathfrak{w}}
\newcommand{\hh}{\mathfrak{h}}
\newcommand{\lieh}{\mathfrak{h}}

\newcommand{\ggl}{\mathfrak{gl}}
\newcommand{\liesl}{\mathfrak{sl}}

\newcommand{\cc}{{\mathfrak{c}}}

\newcommand{\CHc}{{\overset{\circ}{CH}}}

\newcommand{\OO}{\mathcal{O}}
\newcommand{\NN}{\mathbb{N}}

\newcommand{\QQ}{\mathbb{Q}}
\newcommand{\ZZ}{\mathbb{Z}}
\newcommand{\C}{\mathbb{C}}
\newcommand{\Cs}{\mathbb{C}^{\times}}

\newcommand{\K}{{\Bbbk}}
\newcommand{\Ks}{\K^{\times}}

\newcommand{\Ker}{\operatorname{Ker}}

\newcommand{\Zm}{\ZZ_m}
\newcommand{\Zmo}{\ZZ_{m_0}}

\newcommand{\alter}[1]{\left\{\begin{array}{l l}#1\end{array}\right.}
\newcommand{\cv}{\cc^{\vee}}

\newcommand{\ubar}[1]{\underaccent{\bar}{#1}}
\newcommand{\Xnko}{X_n^{(\ell)}\left(\ubar{1}\right)}

\theoremstyle{plain}
\newtheorem{theorem}{Theorem}[section]
\newtheorem{lemma}[theorem]{Lemma}
\newtheorem{proposition}[theorem]{Proposition}
\newtheorem{corollary}[theorem]{Corollary}

\newtheorem{conj}[theorem]{Conjecture}

\theoremstyle{definition}
\newtheorem{definition}[theorem]{Definition}

\newtheorem*{Ack}{Acknowledgements}

\theoremstyle{remark}
\newtheorem{remark}[theorem]{Remark}

\begin{document}

\title[Moment maps for tori and $\theta$]{On the normality of the null-fiber of the moment map for $\theta$- and tori representations}

\author{Michael BULOIS}

\maketitle
\thispagestyle{empty}


\setcounter{tocdepth}{1}

\begin{abstract}
Let $(G,V)$ be a representation with either $G$ a torus or $(G,V)$ a locally free stable $\theta$-representation. We study the fiber at $0$ of the associated moment map, which is a commuting variety in the latter case. We characterize the cases where this fiber is normal. The quotient (\emph{i.e.} the symplectic reduction) turns out to be a specific orbifold when the representation is polar. In the torus case, this confirms a conjecture stated by C. Lehn, M. Lehn, R. Terpereau and the author in a former article. In the $\theta$-case, the conjecture was already known but this approach yield another proof.\end{abstract}

\tableofcontents

\vspace{-6mm}

\section*{Introduction}
Let $\K$ be an algebraically closed field of characteristic zero. Let $(G,V)$ be a representation of a connected reductive algebraic group $G$ on a finite dimensional $\K$-vector space $V$. Let $\gg$ be the Lie algebra of $G$. We define the moment map \[\mu: \left\{\begin{array}{ c c c}V\oplus V^*&\rightarrow& \gg^*\\ (x,\varphi)&\mapsto &(g\mapsto \varphi(g\cdot x))\end{array}\right.\]

Among the fibers of the moment map, the most special one is $\muz$. 
When the representation $(G,V)$ is the adjoint action of $G$ on $\gg$, $\muz$ can be identified wih the commuting scheme of $\gg$. According to a long-standing conjecture, this scheme should be normal. One of the aim of this work is to study the normality of the scheme $\muz$ in some nice cases. For instance, if $(G,V)$ is visible and locally free, we already know that $\muz$ is a complete intersection and the normality question reduces to singular locus questions thanks to Serre's criterion (see \emph{e.g.} Corollary~\ref{cor_locfree} and Remark~\ref{rk_F0F1}). Using this, Panyushev \cite{Pan94} has shown that $\muz$ is reduced and normal in the case of the representation associated to a symmetric Lie algebra of maximal rank. 

There are some natural generalizations of adjoint representations and symmetric Lie algebras: the so-called $\theta$-representations ( or $\theta$-groups) of Vinberg \cite{Vin76} and, more generally, the polar representations of Dadok and Kac \cite{DK85}. Recall that a $\theta$-representation is a representation $(G,V)$ isomorphic to some $(H_0, \lieh_1)$ where $H$ is a connected reductive group acting on its Lie algebra $\lieh$, with $\lieh$  equipped with a $\Zm$-grading $\bigoplus_{i\in \Zm} \lieh_i$ and where $H_0$ is the connected subgroup of $H$ whose Lie algebra is $\lieh_0$.
It turns out that the geometry of $\muz$ can be a bit more involved in this setting of $\theta$-representations than for the classical adjoint representation. For instance, it is possible to find examples of such representations where $\muz$ is non-irreducible \cite{Pan94, Bu11}, non-reduced or irreducible non-normal \cite{BLLT}. In this paper, we classify the locally free stable $\theta$-representations for which $\muz$ is normal. It turns out (Theorem~\ref{main_thm_theta}) that this normality property holds in all the classical cases and we find only $5$ non-normal examples of exceptional cases, each of rank $1$. We take the opportunity to generalize some properties of decompositions classes to the  $\theta$-cases in Section \ref{thetaGroups}.
 The $\theta$-case is partly based on the study of some specific tori representations. This is one motivation for studying $\muz$ for a general representation of a torus in Section~\ref{norm_muz_tori}. The fiber turns out to be normal as soon as it is irreducible and we give a combinatorial criterion for this in terms of the weights of the representation (Theorem~\ref{irrcomp}).

The variety $\muz$ is also involved in the symplectic reduction $(V\oplus V^*)\SPR G:=\muz\GIT G$ where $X \GIT G$ denotes the categorical quotient of an affine $G$-variety $X$. If $\muz$ is normal, then the symplectic reduction is also normal. 
Following a theorem of Joseph \cite{Jos97} in the Lie algebra case, and a conjecture stated in \cite{BLLT}  in the complex case (see conjecture \ref{conjA} here), the variety $(V\oplus V^*)\SPR G$ should be isomorphic to a specific orbifold whenever $(G,V)$ is visible and polar. The results of the present paper yield a proof of this conjecture when $G$ is a torus (Corollary \ref{conjA_tori}), and allows to recover it when $(G,V)$ is a stable locally free $\theta$-representation (Remark \ref{norm_quot}). In the torus case, we also get that that the visibility assumption is necessary (Proposition~\ref{bijection}) but that the scheme $(V\oplus V^*)\SPR G$ is nevertheless reduced (Theorem~\ref{irrcomp}) and normal (Proposition \ref{normal_quot}) in general.

One should note that, as the torus case suggests, we put light on representations which are far from being irreducible. 
Part of the results concerning tori (\emph{e.g.}, the normality of the symplectic reduction) were written independantly in the recent preprint \cite{HSS}. The general philosophy of \cite{HSS} is to focus on ``large'' representations (ex: $2$-large) whereas the polar representations are ``small''.
\begin{Ack}
Section \ref{thetaGroups} of this paper was essentially included in early versions of \cite{BLLT}. I want to thank my co-authors to allow me to reuse this material here. 
\end{Ack}

\section{Generalities}\label{gener}
In this section, we introduce some classical material to study null fibers of moment maps.  Even if some results are stated in a greater generalities, most of this section is heavily inspired by \cite{Pan94}. Popov's article \cite{Po08} is also a source of inspiration.

We work over an algebraically closed field $\K$ of characteristic zero. We use the language of scheme even if these schemes will always be separated of finite type over $\K$ and will often be reduced, i.e. varieties. We do not require that the varieties are irreducible. When speaking about the \textit{irreducible components} of a scheme, we always mean the irreducible components of the corresponding reduced variety and we do not consider embedded components.

If $A$ is a subset of a vector space, $\langle A\rangle$ stands for the subspace generated by $A$.

\begin{definition}
An element $x\in V$ is said to be 
\begin{itemize}
\item semisimple if $G.x$ is closed,
\item nilpotent if $0\in \overline{G.x}.$
\end{itemize}
\end{definition}
Equivalently, $x$ is nilpotent if and only if $\pi(x)=\pi(0)$ where $\pi:V\rightarrow V\GIT G$ is the categorical quotient map.

\begin{definition}
The representation $(G,V)$, is said to be
\begin{itemize}
\item \emph{locally free}, if $\dim G.x=\dim G$ for some $x\in V$. 
\item \emph{visible}, if there are finitely many nilpotent $G$-orbits.
\item \emph{stable}, if there is an open subset of $V$, consisting of semisimple elements
\item \emph{polar}, if there exists $v\in V$ such that $\dim \cc_v=\dim V//G$ where $\cc_v:=\{x\in ~V~|\, \gg\cdot x\subset \gg\cdot v\}$. We then say that $\cc_v$ is a \emph{Cartan subspace} of $V$.
\end{itemize}
\end{definition}

Examples of visible and stable representations are given by the adjoint action of $G$ on $\gg$, or a $\theta$-representation arising from a $\ZZ_2$-grading (so-called \emph{symmetric Lie algebras}). In general, $\theta$-representations are always visible but not necessarily stable \cite{Vin76}. Polar representations, as introduced in \cite{DK85}, can be neither stable nor visible. Examples of locally free representations are given by symmetric Lie algebras \emph{of maximal rank}. 
It is worth noting that no adjoint representation $(G,\gg)$ is locally free. 

When $k=\C$ and $(G,V)$ is polar, we write $\cc$ and $\cv$ for dual Cartan subspaces of $V$ and $V^*$, following \cite[(3.4), (3.6)]{BLLT}, together with dual decompositions $V=\mathfrak c\oplus (\gg\cdot \cc)\oplus U$ and $V^*=\cv\oplus (\gg\cdot \cv)\oplus U^{\vee}$. Write $W$ for the Weyl group $N_{G}(\cc)/C_G(\cc)$. The following was Conjecture $A$ in \cite{BLLT}.
\begin{conj}\label{conjA}
Assume that $(G,V)$ is visible and polar, then we have an isomorphism of Poisson varieties $(V\oplus V^*)\SPR G\cong (\cc\oplus \cv)/W$.\end{conj}
\begin{remark}\label{rk_conjA}
Assume $\K=\C$. From \cite[Proposition~3.3\&3.5]{BLLT}, we have a injective natural morphism $\cc\times\cv/W\rightarrow \muz\GIT G$ of Poisson varieties. Since $\cc\oplus\cv/W$ is normal, Conjecture \ref{conjA} is equivalent to:
\begin{enumerate}
\item The morphism $\cc\times\cv/W\rightarrow \muz\GIT G$ is dominant (i.e. it is bijective.) 
\item $\muz\GIT G$ is normal
\end{enumerate}
\end{remark}

When studying the null fiber of the moment map, it is enlightening to study the modality of $(G,V)$. The modality of a $G$-subvariety $X\subseteq V$ is defined as
\[\mod(G,X):=\max_{Y\subseteq X} \left(\min_{x\in Y} \left(\codim_Y G.x\right)\right)\]
where the maximum is taken over all irreducible subvarieties $Y\subseteq X$. 

In the remaining of the section we will consider a finite cover of $V$, $V=\bigcup_i J_i$ such that
\begin{itemize}
\item  Each $J_i$ is $G$-stable, irreducible and locally closed in $V$ of constant orbit dimension (that is $\exists m_i\in \NN$, $\forall x\in J_i$, $\dim G\cdot x=m_i$),
\item $\overline{J_i}=\overline{J_j}\Rightarrow i=j$. 
\end{itemize}
%
Exemples of such cover includes sheets of $(G,V)$ or decomposition classes in the case of the adjoint action. Note that if the cover is disjoint $V=\bigsqcup_i J_i$, then the second condition follows from the first one.
Since the $J_i$ are of constant orbit dimension, we have $\mod(G,J_i)=\dim J_i-m_i$. Moreover, from finiteness and locally closeness, each irreducible subvariety $Y\subset V$ intersects at least one $J_i$ as a dense open subset of $Y$  and we have $\min_{x\in Y} \codim_Y G.x=\mod(G,J_i)$. Hence \begin{equation}\label{modTVJ}\mod(G,V)=\max_{i} \,\mod(G,J_i).\end{equation} Given a subvariety $X\subseteq V$, we write \begin{equation}\muzx{X}:=\{(x,\varphi)\in \mu^{-1}(0)| x\in X\}=\{(x,\varphi)\in X\times V^*\, |\, \varphi\in (\gg.x)^{\perp}\}\label{def_muzx}\end{equation}
seen as a subvariety of $\muz$.
This allows for instance to decompose $\muz$, as a finite union: \begin{equation}\label{decmuzmuzS}\muz=\bigcup_{i}\muzx{J_i}.\end{equation}

\begin{lemma}\label{irr_comp_muz}
\begin{enumerate}[label=(\roman*)]
\item 
Each $\muzx{J_i}$ is an irreducible and locally closed subvariety of $\muz$ of dimension $\dim V+\mod(G,J_i)$.
\item Each irreducible component of $\muz$ is of the form $\overline{\muzx{J_i}}$ for some $i$. 
\item There is a bijection between the set of irreducible components of $\muz$ of maximal dimension $\dim V+\mod(G,V)$ and the set of $J_i$ with maximal modality $\mod(G,J_i)=\mod(G,V)$.
\end{enumerate}
\end{lemma}
\begin{proof}
(i) From \eqref{def_muzx}, $\muzx{J_i}$ is a subbundle in $J_i\times V^*$ of codimension $m_i$. In particular, it is an irreducible subvariety of $\muz$ of dimension \begin{equation}\dim V+\dim J_i-m_i=\dim V+\mod(G,J_i).\label{dim_muzS}\end{equation}
(ii) is clear from (i) and \eqref{decmuzmuzS}.\\ 
(iii) By considering projection on the first variable, $\overline{\muzx{J_{i}}}=\overline{\muzx{J_{j}}}$ can happen only if $\overline{J_{i}}=\overline{J_{j}}$, that is only if $i=j$. The bijection is then clear from (i) and (ii).
\end{proof}

Another classical result is the following characterization of the smooth locus of $\muz$. If $y$ belongs to a $G$-module $Y$ (e.g. $Y=V$ or $Y=V\oplus V^*$), we let $\gg^y:=\{g\in \gg|\, g.y=0\}$ be the centralizer of $y$ in $\gg$.
\begin{proposition}\label{smooth_muz}
The following assertions are equivalent for an element $(x,\varphi)\in\muz\subset V\oplus V^*$
\begin{enumerate}[label=(\roman*)]
\item  $(x,\varphi)$ is a smooth point of the scheme $\muz$
\item $(x,\varphi)$ belongs to a unique irreducible component $\overline{\muzx{J_i}}$ for some $J_i$, and 
$\dim \gg^{x,\varphi}=\mod(G,J_i)+\dim \gg-\dim V$
\end{enumerate}
\end{proposition} 

\begin{proof}
The tangent space of $\muz$ at $z=(x,\varphi)$ is $\Ker d_{z}\mu$. The dual map of $d_z\mu$ has the following expression 
\[(d_{z}\mu)^*:\left\{\begin{array}{r c l}\gg&\rightarrow &(V\oplus V^*)^*\\ g&\mapsto &\left[(y,\psi)\mapsto \psi(g. x)+\varphi(g. y)\right] \end{array}\right.\]
In particular, $\Ker (d_z(\mu))^*=\{g\in \gg|\, g\cdot x=0=g\cdot \varphi\}=\gg^{x,\varphi}$. Hence $\dim \Ker d_z \mu=2\dim V-\dim \gg+\dim \gg^{x,\varphi}$.

Recall that smooth points on a scheme belong to a unique irreducible component. Let $(x,\varphi)$ belong to a unique irreducible component $\overline{\muzx{J_i}}$ for some $J_i$. Since $\dim \muzx{J_i}=\dim V+\mod(G,J_i)$ (Lemma \ref{irr_comp_muz}), we get that $(x,\varphi)$ is a smooth point of $\muz$ if and only if $\dim V-\dim \gg+\dim \gg^{x,\varphi}=\mod(G,J_i)$, hence the equivalence.
\end{proof}

\begin{corollary}\label{cor_locfree}
Assume that $\mod(G,V)=\dim V-\dim G$. Then 
\begin{enumerate}[label=(\roman*)]
\item $(G,V)$ is locally free.
\item $\muz$ is a complete intersection of dimension $2\dim V-\dim G$.
\item The smooth points of the scheme $\muz$ are the points $(x,\varphi)$ satisfying $\gg^{x,\varphi}=\{0\}$.
\item $\muz$ is irreducible if and only if $\mod(G,J_i)<\mod(G,V)$, for any $i$ such that $\overline{J_i}\neq V$. If it is the case, the scheme $\muz$ is also reduced
\item $\muz$ is reduced and normal if and only if it is irreducible and the following holds
\begin{equation} \label{carac_smooth_codim1}\forall i, \; \left[\mod(G,J_i)=\mod(G,V)-1 \Rightarrow \left(\exists (x,\varphi) \in \muzx{J_i}\textrm{ s.t. } \gg^{x,\varphi}=\{0\}\right)\right]
\end{equation}
\end{enumerate}
\end{corollary}
\begin{proof}
(i) Let $i_0$ be the index such that $\overline{J_{i_0}}=V$. Then $m_{i_0}$ is the maximum orbit dimension in $V$ and we get $\mod(G,V)\geqslant \mod(G,J_{i_0})=\dim V-m_{i_0}\geqslant \dim V-\dim G$. From the hypothesis, we have equalities so $m_{i_0}=\dim G$. \\
(ii) The irreducible components of maximal dimension of $\muz$ have codimension $2\dim V-(\dim V+\mod(G,V))=\dim G$ in $V\oplus V^*$. Since $\muz$ is defined in $V\oplus V^*$ by $\dim \gg$ equation, (ii) follows. \\
(iii) From (ii), $\muz$ is equidimensional and its irreducible components have dimension $2\dim V-\dim G$. 
From proof of Proposition \ref{smooth_muz}, we see that the tangent space of $\muz$ at $(x,\varphi)$ has dimension $2\dim V-\dim \gg+\dim \gg^{x,\varphi}$. The result follows.\\
(iv) We have already seen in (i) that $\mod(G,J_{i_0})=\mod(G,V)$ for $i_0$ such that $\overline{J_{i_0}}=V$. The characterization of the irreducibility of $\muz$ then follows from equidimensionality and Lemma \ref{irr_comp_muz} (iii). For $x\in J_{i_0}$, we have $\gg^x=\{0\}$. So $\muzx{J_{i_0}}$ is in the smooth locus of $\muz$ by (iii).  Hence $\muz$ is reduced as soon as it is irreducible. \\
(v) Recall now that a variety is normal if and only if it satisfies Serre's condition ($S_2$) and its singular locus has codimension at least $2$. The former is provided by (ii). Under the irreducibility assumption, the latter is equivalent to \eqref{carac_smooth_codim1}. Indeed, the singular locus of $\muz$ is a closed subset of $\bigcup_{i\neq i_0}\muzx{J_i}$. But, it follows from Lemma \ref{irr_comp_muz}(i) that an irreducible component of a closed subset of $\bigcup_{i\neq i_0}\muzx{J_i}$ of codimension $1$ in $\muz$ contains at least one $\muzx{J_i}$ with $\mod(G,J_i)=\mod(G,V)-1$. 
\end{proof}

\begin{remark} \label{rk_F0F1}In the notation of the proof of (i), we have, for a general representation $(G,V)$: $\mod(G,V)\geqslant \dim V-m_{i_0}$. The limit case $\mod(G,V)=\dim V-m_{i_0}$ corresponds to condition $(\mathcal F_0)$ of \cite{Pan94}, see (2.2) and Corollary 2.5 in \textit{loc. cit.} Then Proposition 3.1 in \textit{loc. cit.} asserts that this condition hold when $(G,V)$ is visible. In particular, hypothesis of the above corollary are satisfied when $(G,V)$ is locally free and visible.\\
Similarily, condition $(\mathcal F_1)$ in \emph{loc. cit.} is equivalent to the condition appearing in Corollary \ref{cor_locfree} (iv). Related statements to corollary \ref{cor_locfree} can be found in Theorems 2.4, 3.1 and 3.2 in \emph{loc. cit.}
\end{remark}

\section{Null-fiber of the moment map - torus case}
\label{norm_muz_tori}
In this section $V$ is an $n$-dimensional representation of an $r$-dimensional torus $G=T\cong (\Ks)^r$. The weight space $X^*(T)$ is isomorphic to $\ZZ^r$. We let $W_{\QQ}:=X^{*}(T)\otimes_{\ZZ}\QQ$. It is a $\QQ$-vector space of dimension $r$. Recall that the irreducible representations of $T$ are one-dimensional. So we can decompose $V=\bigoplus_{i=1}^n V_{i}$, with $T$ acting on $V_i\cong \K$ with weight $s_i=\left(s_i^j\right)_{j\in[\![1,r]\!]}\in X^*(T)$. That is, the action of an element $t=(t_j)_{j\in [\![1,r]\!]}\in T$ on an element $v=(v_i)_{i\in[\![1,n]\!]} \in V$ is given by \[t\cdot v=\left(\left(\prod_{j=1}^r t_j^{(s_i^j)}\right)v_i\right)_i.\]
We introduce the $n\times r$ matrix $S:=(s_i^j)_{i,j}$. Given $I\subset [\![1,n]\!]$ or $i\in [\![1,n]\!]$, 
we also define the $\#I\times r$-matrix $S_I:=(s_i^j)_{(i,j)\in I\times[\![1,r]\!]}$ and the $(n-1)\times r$-matrix $S_{\hat \imath}:=S_{[\![1,n]\!]\setminus\{i\}}$.


In the sequel we need a partition of $V$ by suitable strata. These are indexed by subsets $I\subset[\![1,n]\!]$ and are given by
\[J_I:=\{(v_i)_i\in V|\, v_i\neq0\Leftrightarrow i\in I\}.\]
\begin{remark}
Each $J_I$ is a locally closed irreducible $T$-stable subvariety whose closure is $V_I:=\{(v_i)_i\in V|\, \forall i\notin I, v_i=0\}$.\\
Moreover, there is a $T$-equivariant action of $(\Cs)^n$ on $V$ given by $(\lambda_i)_i\cdot (v_i)_i=(\lambda_iv_i)_i$. The orbits of this action are precisely the strata $J_I$. Hence many geometric properties of the action of $T$ on a point $v\in V$ (dimension of orbit, nilpotency, semisimplicity, \dots) are preserved along the stratum $J_I$ containing $v$. It will then be convenient to make the computations only at $v_I:=(\delta_{i\in I})_i$ where $\delta$ is the Kronecker symbol. \\
It is now clear that $V=\bigsqcup_I J_I$ is a cover satisfying the assumptions of Section \ref{gener}. 
\end{remark}
For $v\in V$, writing $X_v:=\scalebox{0.7}{$\begin{pmatrix}v_1& &\\&\!\!\ddots\!\!&\\& & v_n\end{pmatrix}$}$, we have
\begin{equation}\tf\cdot v=\langle (s_i^jv_i)_{i\in[\![1,n]\!]}| \, j\in[\![1,r]\!]\rangle=\Im(X_vS)=X_v\Im(S).\label{tspace}\end{equation}
where $\Im(S)$ (resp. $\Im(X_vS)$) denote the image space of the linear maps $\K^r\rightarrow\K^n$ associated to $S$ (resp. $X_vS$) in the canonical basis.
It follows from \eqref{tspace} that, for $v\in J_I$, we have 
\begin{equation}\label{tspaceI}\dim T.v=\rk S_I.\end{equation} In particular,  the modality of $J_I$ is given by 
\begin{equation}\label{calcmodJI}\mod(T,J_I)=\# I-\rk S_I=\dim \Ker({}^tS_I).\end{equation}
where $\Ker({}^tS_I)$ is the kernel of the linear map $\K^{\#I}\rightarrow \K^r$ associated to ${}^tS_I$. Via the natural identifications $\K^{\#I}\cong V_I\subset V\cong \K^n$, $\Ker({}^tS_I)$ is identified with $\Ker({}^tS)\cap V_I$. Assume that $I_1\subset I_2$ are two subsets of $[\![1,n]\!]$, then $\Ker({}^tS)\cap V_{I_1} \subset \Ker({}^tS)\cap V_{I_2}$. So \begin{equation}\label{modJI1I2}I_1\subset I_2\;\;\Rightarrow \;\;\mod(T,J_{I_1})\leqslant \mod(T,J_{I_2}), 
\end{equation}
From \eqref{modTVJ}, we get
\begin{equation}
\label{modTV}
\mod(T,V)=\mod(T, J_{[\![1,n]\!]})=\dim V-\rk S.
\end{equation}

We can apply \eqref{def_muzx} to $X=J_I$ to define $\muzx{J_I}$. This yields a partition of $\muz$ into a disjoint union of irreducible locally closed (Lemma~\ref{irr_comp_muz}) subsets \[\muz=\bigsqcup_I\; \muzx{J_I}\]

\begin{theorem}\label{irrcomp}Assume that $(T,V)$ is a reprentation with $T$ a torus.\\
Then, under above notation:
\begin{enumerate}[label=(\roman*)]
\item $\muz$ is a reduced complete intersection of dimension $2\dim V-\rk S$.  
\item The irreducible components of $\muz$ are the subsets of the form $\overline{\muzx{J_I}}$ with $\rk S-\rk S_I=n-\#I$.
\item $\muz$  is irreducible if and only if, for any $i\in[\![1,n]\!]$, we have $\rk(S_{\hat\imath})=\rk S$.
\item $\muz$ is normal if and only if it is irreducible.
\end{enumerate}
\end{theorem}
\begin{proof}

By Lemma \ref{irr_comp_muz}, the irreducible components of $\muz$ of maximal dimension $\dim V+\mod(T,V)=2\dim V-\rk S$ \eqref{modTV} are the $\overline{\muzx{J_I}}$ with $\#I-\rk S_I=n-\rk S$ \eqref{calcmodJI}. 
This last condition is the one expressed in (ii).

Then, from \eqref{modJI1I2}, $\muz$ has a single irreducible component of maximal dimension if and only if $\mod(T,J_I)<\mod(T,J_{[\![1,n]\!]})$ for any $I$ of the form $[\![1,n]\!]\setminus\{i\}$. If $I=[\![1,n]\!]\setminus\{i\}$, it follows from \eqref{calcmodJI}, that \[\mod(T,J_I)=\left\{\begin{array}{l l}\mod(T,J_{[\![1,n]\!]})-1& \textrm{ if } \rk S_{\hat\imath}=\rk S\\
\mod(T,J_{[\![1,n]\!]})& \textrm{ if } \rk S_{\hat\imath}=\rk S-1 \end{array}\right..\]

In the next paragraph, we show that $\muz$ is a complete intersection, hence equidimensional. Together with the above arguments, this implies that (ii) and (iii) holds.

Consider the kernel of the representation $(T,V)$, $K:=\{t\in T| \, \forall i, \,s_i(t)=1\}$. Its Lie algebra is $\kk:=\{t\in \tf\cong \K^r|S t=0\}=\Ker(S)$. Applying \eqref{tspaceI} to a general element $v\in V$, we get $\dim (T/K)\cdot v=\dim T\cdot v=\rk S=\dim \tf-\dim \kk$. So the representation $(T/K,V)$ is not only faithul but locally free. Note that $T/K$ is also a torus and that $\muz$ is the same for $(T,V)$ and $(T/K,V)$. Therefore, there is no loss of generality in assuming that $(T,V)$ is locally free in order to prove (i) and (iv). Under this assumption, important features are $\rk S=\dim T$ \eqref{tspaceI} and $\mod(T,V)=\dim V-\dim T$ \eqref{modTV}. In particular, we can apply Corollary \ref{cor_locfree} and get that $\muz$ is a complete intersection.

Assuming that $(T,V)$ is locally free, we claim that \begin{equation}\label{smoothloc} \forall I\subset [\![1,n]\!],\; \exists (x,\varphi)\in \muzx{J_I},\; \tf^{(x,\varphi)}=\{0\}.\end{equation}
Under this claim, (ii) and Corollary \ref{cor_locfree} (iii) imply that each irreducible component of $\muz$ has smooth points. Hence $\muz$ is reduced and this ends the proof of (i). The claim also implies that condition \eqref{carac_smooth_codim1} is automatically satisfied so (iv) follows from Corollary \ref{cor_locfree} (v).

\textit{Proof of claim \eqref{smoothloc}}
Let $I\subset [\![1,n]\!]$. The decomposition $V=\bigoplus_i V_i$ induces to a decomposition $V^*=\bigoplus V_i^*\cong \K^n$. We set $x=v_I=(\delta_{i\in I})_i\in J_i\subset V$ and $\varphi:=(\delta_{i\notin I})_i\in V^*$ with respect to these decompositions. Since $\tf\cdot x\subset V_I$, we have $\varphi\in (\tf\cdot x)^{\perp}$ so $(x,\varphi)\in \muzx{J_I}$. 
Also, let $t\in T$ stabilizing ${(x,\varphi)}$. Then $s_i(t)=1$ for any $i\in I$ and $1/(s_i(t))=1$ for any $i\notin I$. Hence $t$ is in the kernel of the representation. By local freeness, we get $\tf^{x,\varphi}=\{0\}$.
\end{proof}

%

\section{Symplectic reduction for tori representations}
\label{norm_symp_tori}
We keep the notation and setting of the previous section. The aim of this section is to study the symplectic reduction $(V\oplus V^*)\SPR T:=\muz\GIT T$. We prove Conjecture~\ref{conjA} in the context of representations of tori (Corollary~\ref{conjA_tori}). Moreover, we show that that the conjecture is essentially false for non-visible action (Proposition~\ref{bijection}), but that the symplectic reduction remains normal in this case (Proposition~\ref{normal_quot}).

In this section, many properties depend on considerations on the convex hull of some of the weights $s_i\in X^*(T)$. If $\emptyset \neq A\subset X^*(T)\subset  W_{\QQ}=X^*(T)\otimes_{\ZZ}\QQ$, we denote by $CH(A)$, the convex hull of $A$ in $W_{\QQ}$. We also denote by $\CHc(A)$ the interior (for the classical topology) of $CH(A)$ in $\langle A \rangle$. Alternatively, $\CHc(A)$ is the set of convex combination $\sum_{a\in A}\alpha_aa$ with each $\alpha_a$ positive. It is an elementary result on convex hull that $0\notin CH(A)$ (resp. $0\notin \CHc(A)$) if and only if $\exists \phi\in W_{\QQ}^*$ s. t. $\forall\, a\in A$, $\phi(a)>0$ (resp. $\phi(a)\geqslant 0$ and $\phi(A)\neq\{0\}$). Note that such $\phi$ can be multiplied by a sufficiently large integer so that we can always assume $\varphi\in X^*(T)$. In our context of weights on a group, this translates as follows.
\begin{lemma}\label{caracCH}
Under above notation, $0\notin CH(A)$ (resp. $0\notin \CHc(A)$) if and only if there exists a one-parameter subgroup $\Cs\stackrel{\rho}{\hookrightarrow} T$ such that $\forall a\in A$, $\exists b_a\in \NN^*$ (resp. $b_a\in \NN$ with at least one $a\in A$ s.t. $b_a\neq0$), $\forall t\in \Cs$, $a(\rho(t))=t^{b_a}$.
\end{lemma}

\begin{proposition} \label{nilpss} Let $\emptyset\neq I\subseteq [\![1,n]\!]$ and $v\in J_I$. 
\begin{enumerate}
\item $v\in V$ is nilpotent if and only if $0\notin CH(\{s_i|i\in I\})$.
\item $v\in V$ is semisimple if and only if $0\in \CHc(\{s_i|i\in I\})$.
\end{enumerate}
\end{proposition}
\begin{proof}
First, recall that the elements of $J_I$ all share the same orbit dimension. In particular, an element $y\in \overline{T.v}\setminus T.v$ lies in $J_{I'}$ for some $I'\subsetneq I$.
Then, from Hilbert-Mumford criterion, $v=(v_i)_i\in V\setminus \{0\}$ is nilpotent (resp. is not semi-simple) if and only if there exists a one-parameter subgroup $\Cs\stackrel{\rho}{\hookrightarrow} T$ such that $\rho(t).v_i=t^{b_i}v_i$ with  $b_i$ positive for each $i\in I$ (resp. non-negative for each $i\in I$ and at least one $b_i$ non-zero) . Since $t^{b_i}=s_i(\rho(t))$, the result then follows from Lemma \ref{caracCH}.
\end{proof}

\begin{corollary}\label{carac_stable}
The representation $(T,V)$ is stable if and only if $0\in \CHc(\{s_i|i\in [\![1,n]\!]\})$.
\end{corollary}


Given $(T,V)$ and associated weights $(s_i)_{i\in [\![1,n]\!]}$, we split $[\![1,n]\!]$ into \begin{equation}I_d:=\{i\in [\![1,n]\!]| \, \rk S_{\hat\imath}=\rk S\}, \quad I_f:=I\setminus I_d=\{i\in [\![1,n]\!]| \, \rk S_{\hat\imath}=\rk S-1\}.\label{def_If}\end{equation}

\begin{proposition}\label{normal_quot}
There is an isomophism $(V\oplus V^*)\SPR T\cong (V_{I_d}\oplus V_{I_d}^*)\SPR T$. \\ In particular, $(V\oplus V^*)\SPR T$ is normal. 
\end{proposition}
\begin{proof}
We denote by $\mu_{I_d}$ the moment map $V_{I_d}\oplus V_{I_d}^*\rightarrow \tf^*$. Recall that $\muz$ and $\mu_{I_d}^{-1}(0)$ are reduced by Theorem \ref{irrcomp}. Hence so are $(V\oplus V^*)\SPR T=\muz\GIT T$ and $(V_{I_d}\oplus V_{I_d}^*)\SPR T=\mu_{I_d}^{-1}(0)\GIT T$.  Note that $V_{I_d}^*$ can be identified with the subrepresentation of $V^*$ generated by the isotypical components associated to the weights $\{-s_i|i\in I_d\}$. This allows us to consider $\mu_{I_d}^{-1}(0)$ as a $T$-stable closed subscheme of $\mu^{-1}(0)$. Therefore, we have a closed immersion $\mu_{I_d}^{-1}(0)\GIT T\rightarrow \mu^{-1}(0)\GIT T$. To show that this is an isomorphism, there remains to prove that this map is a surjective. For this, we are going to show that the only closed orbits in $\muz$ lie in $\mu_{I_d}^{-1}(0)$.

Let $(x,y)\in \mu^{-1}(0)$ and assume that $x\notin V_{I_d}$. From the assumption on $x$, we know that there exists $i$ such that $x_i\neq0$ and $\rk S_{\hat\imath}<\rk S$. 
Then there exists an element $ t\in \langle s_j|\, j\neq i\rangle^{\perp}\subset \tf$ such that $s_i(t)=1$. So $t\cdot x=x_i$ and $V_{\{i\}}\subset \tf\cdot x$. Since $y\in (\tf\cdot x)^{\perp}$, we get $y_i=0$. Then a one-parameter subgroup $\Cs\stackrel{\rho}{\hookrightarrow} T$ given by $s_j(\rho(t))=t^{\delta_{i,j}}$ acts on $(x,y)$ via $t\cdot y=y$ and $\lim_{t\rightarrow 0} t\cdot x=x'$ with $x'_j=x_j$ if $j\neq i$ and $x'_i=0$. Hence $(x',y)\in \overline{T.(x,y)}\setminus T.(x,y)$. A similar argument shows that $T.(x,y)$ is also non-closed as soon as $y\notin V_{I_d}^*$.

Let us now show that $(V_{I_d}\oplus V_{I_d}^*)\SPR T$ is normal. For each $i\in I_f$, $s_i$ is not a linear combination of the other $s_j$ ($j\neq i$). Hence $\{s_i\,|\,i\in I_f\}$ induces a basis of $\langle s_i\,|\,i\in [\![1,n]\!]\rangle/\langle s_i\,|\,i\in I_d\rangle$ and $\rk S_{I_d}=\rk S-\#I_f$. Assuming that $\rk S_{I_d\setminus\{i\}}=\rk S_{I_d}-1$ for some $i\in I_d$, we would then have $\rk S_{\hat\imath}\leqslant \rk S_{I_d}-1+\#(I_f)=\rk S-1$, which contradicts the assumption $i\in I_d$. So $\rk S_{I_d\setminus\{i\}}=\rk S_{I_d}$ for any $i\in I_d$ and $(T,V_{I_d})$ satisfies hypothesis of Theorem \ref{irrcomp} (iii). Hence $\mu_{I_d}^{-1}(0)$ is irreducible and normal by Theorem \ref{irrcomp}, so the same holds for its quotient by $T$.
\end{proof}

We have the following useful characterization of a visible action:

\begin{proposition}\label{carac_visible}
The representation $(T,V)$ with weight $(s_i)_{i\in [\![1,n]\!]}$ is visible if and only if there is a partition 
$[\![1,n]\!]=\bigsqcup_{j=0}^l I_j$ satisfying
\begin{enumerate}[label=(\roman*)]
\item $\langle s_i|\, i\in [\![1,n]\!]\rangle=\bigoplus_{j=0}^l \langle s_i|\, i\in I_j\rangle$
\item $\dim \langle s_i|\, i\in I_j\rangle=\left\{\begin{array}{l l} \#I_j \textrm{ if $j=0$}\\\#I_j-1 \textrm{ if $j\geqslant 1$} \end{array}\right.$
\item $\forall j\geqslant 1$, $0\in \CHc(\{s_i|i\in I_j\})$.
\end{enumerate}
\end{proposition}
Note that the subset $I_0$ in the proposition has to coincide with $I_f$ defined in \eqref{def_If} and that $I_d=\bigsqcup_{j\geqslant 1} I_j$.
\begin{proof}
{\bf ``if'' part:} Assume that such a partition exists. Let $x\in V$ be nilpotent and let $\tilde I$ be such that $x\in J_{\tilde I}$. Define $\tilde{I}_j:=\tilde I\cap I_j$. Assume that $j$ is a positive index such that $\tilde{I}_j\neq\emptyset$. Since $x$ is nilpotent and $\tilde{I}_j\subset \tilde I$, it follows from Proposition \ref{nilpss} that $0\notin CH(\{s_i|i\in \tilde{I}_j\})$. By (iii), $\tilde{I}_j$ is a proper subset of $I_j$.

On the other hand, (iii) also implies that there is a convex combination $\sum_{i\in I_j} \alpha_is_i=0$ with each $\alpha_i>0$. By (ii), this combination is the unique linear relation (up to scalar multiplication) between the $(s_i)_{i\in I_j}$. In particular, the family $(s_i)_{i\in \tilde I_j}$ is linearly independant. The same hold if $j=0$ by (ii). By (i), we get that $(s_i)_{i\in \tilde I}$ is linearly independent. 

Then, by \eqref{tspaceI}, $\dim T.x=\rk S_{\tilde I}=\#{\tilde I}=\dim J_{\tilde I}$. Hence $T.x$ is open in $J_{\tilde I}$. Since this holds for any $x\in J_{\tilde I}$, we get $T.x=J_{\tilde I}$. Since there are finitely many possible index set $\tilde I$, the representation $(T,V)$ is visible.

{\bf ``only if'' part:}  
We will argue by induction on  $\dim V$.
Assume $\dim V=1$. If $s_1\neq0$, we set $l=0$ and $I_0:=\{s_1\}$. If $s_1=0$, we set $l:=1$, $I_0:=\emptyset$ and $I_1:=\{s_1\}$.

Assume now that the result holds true for any visible representation of $T$ of dimension $n-1$ and consider a visible representation $(T,V)$ of dimension $n$. Since $V':=V_{[\![1,n-1]\!]}$ is a subrepresentation of $V$, it is also visible. By the inductive hypothesis, there exists a partition $[\![1,n-1]\!]=\bigsqcup_{j=0}^{l'} I'_j$ satisfying (i), (ii) and (iii). If $s_{n}\notin \langle s_i|\, i\in [\![1,n-1]\!]\rangle$, we can set $I_0:=I_0'\cup \{n\}$ and $I_j:=I'_j$ for any $j\geqslant 1$.
 
From now on, we assume that $s_{n}\in \langle s_i|\, i\in [\![1,n-1]\!]\rangle$. As a first step, we claim that there is a subset $\tilde I\subset [\![1,n-1]\!]$ and some non-zero rational coefficients $(\beta_i)_{i\in \tilde I}$ such that 
\begin{itemize}
\item $s_{n}:=\sum_{i\in \tilde I} \beta_i s_i$ 
\item $\tilde I\cap I_j'\neq I_j'$ for any $j\geqslant 1$ 
\item $\beta_i>0$ for $i\in \bigsqcup_{j\geqslant 1} I_j'\cap \tilde I$.
\end{itemize}
Indeed, start with a combination with rational coefficients $s_{n}:=\sum_{i\in [\![1,n-1]\!]} \beta'_i s_i$. Fix $j\geqslant 1$. From (iii) and (ii), there is a combination with positive coefficients $0=\sum_{i\in I_j'} \alpha_i s_i$. Let $\lambda_j:=\max_{i\in I_j'}\{-\beta_i'/\alpha_i\}$. Set $\beta_i:=\beta_i'+\lambda_j \alpha_i$ for $i\in I_j'$. So $\beta_i\geqslant 0$ for any $i\in I_j'$ and $\beta_i=0$ for at least one $i\in I_j'$. Then set $\tilde I:=\{i| \beta_i\neq0\}$ and the claim is shown.

From the property $\tilde I\cap I_j'\neq I_j'$ and (i), (ii), (iii),, we get that $(s_i)_{i\in \tilde I}$ is linearly independant as in the proof of the ``if part''. So, up to scalar multiplication, the only relation between $\{s_i|i\in \tilde I\cup \{n\}\}$ is given by the above coefficients $\beta_i$ ($i\in \tilde I$). Assuming that $\beta_i>0$ for some $i\in \tilde I$ (\emph{e.g.} if $\tilde I\cap \bigsqcup_{j\geqslant 1} I_j'\neq\emptyset$), we get $0\notin CH(\{s_i|i\in \tilde I\cup \{n\}\})$ so orbits in $J_{\tilde I\cup \{n\}}$ are nilpotent by Proposition \ref{nilpss}. From \eqref{tspaceI}, these orbits are of dimension $\rk S_{\tilde I\cup \{n\}}=\#\tilde I=\dim J_{\tilde I\cup \{n\}}-1$. This contradicts the visibility of $(T,V)$. As a result, $\tilde I\subset I_0$ and $\beta_i<0$ for each $i\in I_0$. Setting $l:=l'+1$, $I_{l}:=\tilde I\cup \{n\}$, $I_0:=I_0'\setminus \tilde I$ and $I_j:=I_j'$ for $j\in [\![1,l']\!]$ then yield (i), (ii) and (iii) for the representation $(T,V)$.
 \end{proof}


\begin{corollary} \label{vis_so_pol}
Any visible representation of a torus is polar.
\end{corollary}
\begin{proof}
Let $(T,V)$ be a visible representation and adopt notation of Proposition \ref{carac_visible}. For $\underline{w}:=(w_1, \dots,w_l)\in \C^l$, we consider $x(\underline{w})$ as the element $(x_i)_{i\in [\![1,n]\!]}$ given by $x_i:=0$ if $i\in I_0$ and $x_i:=w_j$ if $i\in I_j$ ($j\geqslant 1$). From \eqref{tspace}, we get $\tf\cdot x(\underline{w})=X_{x(\underline{w})}Im(S)$. From Proposition \ref{carac_visible} (i), $Im(S)=\bigoplus_j Im(S_{I_j})$. Since $X_{x(\underline w)}$ acts by $w_jId$ on $Im(S_{I_j})$, we get that $\tf\cdot x(\underline{w})$ is included in $\tf\cdot x(1,\dots,1)$. The subspace $\cc:=\{x(\underline w)|\underline w\in \C^l\}$ is our candidate for being a Cartan subspace of $(T,V)$ and there remains to show that $l=\dim V\GIT T$, \cite{DK85}.

Since $(T,V)$ is visible, a general fiber of the quotient map $V\rightarrow V\GIT T$ is the closure of a general orbit in $V$. Hence, by \eqref{tspaceI}, $\dim V\GIT T=\dim V-\rk S$. This is equal to $l$ by Proposition \ref{carac_visible}.
\end{proof}

\begin{corollary}\label{st_irr}
Assume that $(T,V)$ is visible. Then the following assertions are equivalent
\begin{itemize} 
\item $(T,V)$ is stable, 
\item $I_f=\emptyset$, 
\item $\mu^{-1}(0)$ is irreducible.
\end{itemize}
\end{corollary}
\begin{proof}
Assume that $(T,V)$ is visible and and adopt notations of Proposition \ref{carac_visible}. Then, by Corollary \ref{carac_stable}, the action is stable if and only if $I_0=\emptyset$. Since $I_0=I_f$, the result follows from Theorem \ref{irrcomp} (iii).
%
\end{proof}

\begin{remark}
Thanks to the previous corollaries, we can re-interpret Proposition \ref{carac_visible} as follows. The visible tori actions are essentially made of two part: a part of rank $0$ provided by $I_f$ and several stable blocks of rank $1$, each looking very much like \cite[Example 8.6]{BLLT}. Another appearence of stable blocks of rank $1$ can be found in Lemma \ref{111}. 
\end{remark}

Let us now ake a closer look at $(V\oplus V^*)\SPR T$, that is at closed orbits of $\mu^{-1}(0)$. A first observation is the following. 
\begin{lemma} \label{lm:ss_s_s}
If $(x,y)\in \mu^{-1}(0)$ is such that $x\in V$ and $y\in V^*$ are semisimple, then $(x,y)$ is semisimple in $V\oplus V^*$.
\end{lemma} 
\begin{proof}
Consider the set of weigths for the action of $T$ on $V\oplus V^*$. Denote by $A_x$ (resp. $A_y$, $A_{x,y}$) the set of weights corresponding to the support of $x$ (resp. $y$, $(x,y)$). By Proposition \ref{nilpss}, $0\in \CHc(A_x)$ and $0\in \CHc(A_y)$. So $0\in \CHc(A_x\cup A_y)=\CHc(A_{x,y})$. The result follows.
\end{proof}

\begin{proposition}\label{bijection}
The representation $(T,V)$ is visible if and only if closed orbits of $\mu^{-1}(0)$ are those of the form $T.(x,y)$ with $T.x$ and $T.y$ closed. 
\end{proposition}
\begin{proof}
$\bullet$ Assume that $(T,V)$ is visible. With similar arguments to those of Proposition \ref{normal_quot}, we are going to show that, whenever $x\in V$ is not semisimple, $(x,y)\in \muz$ cannot be semisimple. A symmetric argument on $y\in V^*$ and Lemma \ref{lm:ss_s_s} then allow to conclude for the ``only if'' part. 

Let $x\in J_{\tilde I}$ for some $\tilde{I}\subset [\![1,n]\!]$. The representation $(T,V_{\tilde{I}})$ is also visible, so we can apply Proposition \ref{carac_visible} (with $\tilde{I}$ instead of $[\![1,n]\!]$) and get a partition $\tilde I=\bigsqcup_{j=0}^l \tilde I_j$. Apply also Proposition \ref{carac_visible} to $V$, so that we get a partition $[\![1,n]\!]=\bigsqcup_j I_{j}$. 

Assuming that $x$ is not semisimple amounts to assume that $\tilde I_0\neq\emptyset$ (Proposition \ref{nilpss}). If $\tilde I_0\cap I_0\neq\emptyset$, Proposition \ref{normal_quot} states that any element of $\muz$ of the form $(x,y)$ is not semisimple. From now on, we assume that there exists $j\geqslant 1$ such that $\tilde I_0\cap I_j\neq\emptyset$. From property (ii) of Proposition \ref{carac_visible} on $\tilde I_0$ and $I_j$, we also know that $I_j \setminus \tilde I_0\neq\emptyset$. Fix $i_0\in \tilde I_0\cap I_j$ and $i_1\in I_j \setminus \tilde I_0$. The relation between the $(s_{i})_{i\in I_j}$ can be written as $\sum_{i\in I_j}\alpha_i s_i=0$ with the $\alpha_i$ positive integers. 
Then we can define a one parameter subgroup $\rho$: 
$\C^{\times}\hookrightarrow T$,  via 
\[s_{i}(\rho(t))=\left\{\begin{array}{l l}  
t^{\alpha_{i_1}} & \textrm{if $i=i_0$}\\
t^{-\alpha_{i_0}} & \textrm{if $i'=i_1$}\\
1 & \textrm{else}\\
\end{array}\right.
\] 
Then $V_{\{i_0\}}=Lie(\rho(\C^{\times}))\cdot x\subset\tf\cdot x$ and, whenever $(x,y)\in \muz$, we have $y_{i_0}=0$. Since $\rho(t)$ acts on $y_{i_1}$ by multiplication by $t^{\alpha_{i_0}}$, we get that $\lim_{t\rightarrow 0}(\rho(t)\cdot(x,y))$ exists and that its first component lies in $J_{\tilde I\setminus\{i_0\}}$. In particular, $T.(x,y)$ is not closed.

$\bullet$ Assume now that $(T,V)$ is not visible. We want to find $(x,y)$ semisimple such that $x$ is not semisimple. Let $I\subset[\![1,n]\!]$ be such that $J_I$ contains infinitely many nilpotent orbits. Then $0\notin CH(\{s_i|i\in I)$ (Proposition \ref{nilpss}) and $\rk S_I<\dim J_I=\#I$ \eqref{tspaceI}. Thus there exists non-trivial linear relations $\sum_{i\in I}\alpha_is_i=0$ with integer coefficients and for any such relation the $\alpha_i$ are not all of the same sign. We can therefore write a relation of the form $\sum_{i\in I'} \beta_i s_i=\sum_{i\in I''} \beta_i s_i$ with each $\beta_i$ positive, $I',I''\neq\emptyset$ and $I'\sqcup I''\subset I$. Let $x=(x_i)_i \in J_{I'}$ and, analogously in $V^*$, let $y=(y_i)_{i} \in J_{I''}\subset V_{I''}^*\subset V^*$. From this, we get that $x$ is non-zero nilpotent, that $y\in V_{I''}^*\subset (V_{I'})^{\perp}\subset (\tf\cdot x)^{\perp}$ and that \[T.(x,y)\subset \left\{(x',y')\in V_{I'}\times V_{I''}^*\left|\,\prod_{i\in I'} (x'_i)^{\beta_i}\prod_{i\in I''} (y'_i)^{\beta_i}=\prod_{i\in I'} (x_i)^{\beta_i}\prod_{i\in I''} (y_i)^{\beta_i}\right\}\right..\] This last subset is a $T$-invariant closed subset of $V$ which does not intersect $\{0\}\times V^*$. Let $(x',y')$ be in the closed orbit of $\overline{T.(x,y)}$. Since $x'\subset J_{\tilde I}$ with $\tilde I\subset I'$, $x'$ is non-zero nilpotent hence not semisimple. 
\end{proof}

Proposition \ref{bijection} shows that the visibility assumption in Conjecture~\ref{conjA} is necessary, since the statement fails set-theoretically in any non-visible case..

When $\K=\C$, we can state the following:
\begin{corollary}\label{conjA_tori}
Conjecture \ref{conjA} holds true for representations of tori
\end{corollary}
\begin{proof}
Note that the polar assumption is redundant by Corollary \ref{vis_so_pol}. 
According to Proposition \ref{bijection}, the generic elements of the irreducible (Proposition \ref{normal_quot}) variety $V\oplus V^*\SPR T$ are closed orbits of the form $T.(x,y)$ with $x$ regular (i.e. with T.x closed of maximal dimension among closed orbits). We may assume that $x\in \cc$. Then $y\in (\tf\cdot x)^{\perp}=(\tf\cdot \cc)^{\perp}=\cv\oplus U^{\vee}$ and, since $T.y$ is closed, we get $y\in \cv$ \cite[Corollary 2.5]{DK85}. Hence (1) of Remark \ref{rk_conjA} holds 
Statement (2) follows from Proposition \ref{normal_quot}.
\end{proof}

\section{Null-fiber of the moment map - $\theta$ case}

\label{thetaGroups}
A standard reference for the theory of $\theta$-representations is \cite{Vin76}. These representations are visible and polar but they are not always stable. A classification of the irreducible $\theta$-representations with their main features can be found in \cite[\S 3]{Kac80}.  

Let us first introduce some notation. Let $\hh=\bigoplus_{i\in \Zm} \hh_i$ be a reductive Lie algebra equipped with a $\ZZ_m$-grading, that is $[\hh_{i},\hh_j]\subset \hh_{i+j}$. Let $H$ be any connected algebraic group having $\hh$ as its Lie algebra. We consider the adjoint representation of $G:=H_0$ on $V:=\hh_1$ where $H_0$ is the connected subgroup of $H$ with Lie algebra $\hh_0$. Fix $\omega$ an $m$-th root of $1$. We denote by $\theta$ the automorphism of $\hh$ given by $\theta(x)=\omega^i x$ for $x\in \hh_i$. If $\af$ is a $\theta$-stable subspace of $\hh$, we can decompose $\af=\bigoplus \af_i$ with $\af_i:=\af\cap\hh_i$. We then denote the dual direct sum by $\af^*=\bigoplus \af_i^{\vee}$. We will identify $\af_i^{\vee}$ and $(\af_i)^*$. 

We need to understand $\muz$. 
From \cite[\S1.2 and \S2.2]{Vin76}, the $\theta$-representation $(H_0,\hh_1)$ is isomorphic to $(\prod_k (H^{(k)})_0,\prod (\hh^{(k)})_1)$ for some graded Lie algebras $\hh^{(k)}=\bigoplus(\hh^{(k)})_i$ which are either simple or abelian. The $((H^{(k)})_0,(\hh^{(k)})_1)$ with $\hh$ simple are called the \emph{simple components} of $(H_0,\hh_1)$. Then $\muz$ is the direct product of the $\mu_{k}^{-1}(0)$ where $\mu_k$ is the moment map associated to the representation $((H^{(k)})_0,(\hh^{(k)})_1)$. In the abelian case, $\muz=\hh_1\oplus\hh_1^*$. From now on, we will assume that $\hh$ is simple.

The Killing form $L(\cdot,\cdot)$ is a non-degenerate symmetric bilinear form on $\hh\times \hh$ which is invariant under $H$ and $\theta$. Via $L$, $V^*\cong \hh_1^{\vee}$ identifies to $\hh_{-1}$. If $x \in \hh$ and $\af\subset \hh$, we denote the centraliser of $x$ in $\af$ by $\af^x=\{ y \in \af \ |\ [x,y]=0\}$. Assume that $\varphi\in V^*$ corresponds to $y\in \hh_{-1}$ via $L$. Then the condition $\varphi\in (\gg\cdot x)^{\perp}$ is read $0=L([\gg,x],y)=L(\hh_0,[x,y])$, that is $[x,y]=0$ since $[x,y]\in \hh_0$ and $L$ is non-degenerate on $\hh_0\oplus\hh_0$. Thus $\muz$ is isomorphic to the commuting scheme:
\begin{equation}\muz\cong \{(x,y)\in \hh_1\times\hh_{-1}|\, [x,y]=0\}.\end{equation}


%

We now want to define a suitable locally closed cover of $\hh_1$ satisfying the conditions stated in Section \ref{gener}
\begin{definition}
We say that a Levi subalgebra $\ll$ of $\hh$ \emph{arises from $\hh_1$} if there exists a semisimple element $s\in \hh_1$ such that $\ll=\hh^s$. 
\end{definition}
In other words, $(\zz_{\ll})_{reg}\cap \hh_1\neq\emptyset$ where $\zz_{\ll}$ is the center of $\ll$ and $(\zz_{\ll})_{reg}$ denotes the set of elements of $\zz_{\ll}$ whose $H$-orbit is of maximal dimension.
Such a Levi $\lf$ is clearly $\theta$-stable. Hence $\lf=\bigoplus \lf_i$ gives rise to a $\Zm$-graded reductive Lie algebra. We denote by $L_0$ the connected subgroup of $H$ with Lie algebra $\lf_0$. 

Recall from \cite{Vin76} that any $x\in\hh_1$ has a Jordan decomposition $x_s+x_n$ with $x_s,x_n\in \hh_1$, the semisimple and nilpotent part of $x$, respectively.
\begin{definition}
Given $\ll$ a Levi subalgebra of $\hh$ arising from $\hh_1$, and  $\OO$ a nilpotent $L_0$-orbit in $\ll_1$, we define the \emph{decomposition class} associated to $(\lf,\OO)$ as 
\[J(\lf,\OO):=H_0\cdot\{(\zz_{\ll})_{reg}\cap \hh_1+\OO\}\]
\end{definition}
Alternatively, $x\in J(\lf,\OO)$ if and only if $x$ is $H_0$-conjugate to an element $x'\in \ll_1$ such that $\lf=\hh^{x'_s}$ and $\OO=L_0\cdot x'_n$. Hence decomposition classes are equivalence classes, so they form a partition of $\hh_1$. Next, we adapt some results of \cite{Bro}  (see also \cite[\S39]{TY}).
\begin{proposition}\label{dec_classes}
Let $J(\lf,\OO)$ be a decomposition class. 
\begin{enumerate}[label=(\roman*)]
\item If $x\in J(\lf,\OO)$, then $J(\lf,\OO)=H_0\cdot \{y\in \hh_1|\, \hh^y=\hh^x\}$
\item $J(\lf,\OO)$ is an irreducible locally closed subset of $\hh_1$.
\item If $x\in J(\lf,\OO)$ then $\dim H_0.x=\dim \hh_1-\dim \lf_1+\dim \OO$. 
\item $\mod(H_0,J(\lf,\OO))=\dim (\zz_{\lf})_1$.
\end{enumerate}
\end{proposition}
\begin{proof}
There is no loss of generality in assuming that $x$ satisfies $\hh^{x_s}=\lf$ and $\OO=L_0\cdot x_n$.

(i) From \cite[35.3.3, 39.1.2, 39.1.1]{TY}, we have $\{y\in \hh_1| \,\hh^y=\hh^x\}=\{y\in \lf_1|\, \hh^{y_s}=\ll, \,(\lf')^{y_n}=(\lf')^{x_n}\}$ where $\lf':=[\lf,\lf]$. Since $L_0$ acts trivially on $(\zz_{\ll})_{reg}$, this yields $J(\lf,\OO)\subset H_0\cdot \{y\in \hh_1| \hh^y=\hh^x\}$.
Note that the condition $\hh^{y_s}=\ll$ implies that $y, y_s,y_n\in\lf_1$. In order to prove the converse inclusion, it is therefore enough to show that $\mathcal U:=\{n\in\lf'_1| \, (\lf')^n=(\lf')^{x_n}\}\subset \OO$.
Let $n\in \mathcal U$. Then $\lf_0^n=\lf_0^{x_{n}}$ so $\dim L_0\cdot n=\dim L_0\cdot x_n=\dim \OO$.
The subset $\mathcal U$ is a dense open subset of $(\lf'_1)^{x_n}$ by \cite[35.3.3]{TY}, hence $\mathcal U$ is irreducible. From \cite[35.3.4]{TY}, $\mathcal U$ is included in the nilpotent cone of $\lf'_1$ which consists of fintely many $L_0$-orbits. Let $\OO'$ be the $L_0$-nilpotent orbit whose intersection with $\mathcal U$ is dense in $\mathcal U$. Then $\mathcal U\subset \overline{\OO'}$. Since all the elements of $\mathcal U$ share the same $L_0$-orbit dimension and $x_n\in \OO\cap \mathcal U$, we can eventually conclude that $\OO=\OO'$ and $\mathcal U\subset \OO$.  
%
%

(ii) It is clear from the definition that $J(\lf,\OO)$ is irreducible. Set $m=\dim \hh^x$ for some $x\in J(\lf,\OO)$. By (i), it does not depend on the choice of $x$. The map $\varphi:y\mapsto \hh^y$, from $\{y\in \hh |\dim \hh^{y}=m\}$ to the Grassmannian of $m$-dimensional subspaces in $\hh$, is a morphism of varieties \cite[19.7.6, 29.3.1]{TY}. By (i), $J(\lf,\OO)=\hh_1\cap \varphi^{-1}(H_0\cdot(\hh^x))$ with $\hh^x$ seen as an element of the Grassmannian. As orbits are locally closed sets, so are their inverse images by morphisms of varieties. Hence $J(\lf,\OO)$ is locally closed in $\{y\in \hh |\dim \hh^{y}=m\}$ and the result follows.

(iii) Consider $\ad x$ as an endomorphism of $\hh$. Then $\lf$ is its generalized eigenspace associated to $0$. Let us denote by $\qq$ the sum of the other generalized eigenspaces. We have $\hh=\lf\oplus\qq$ and $\hh\cdot x=[\lf\oplus\qq, x]=(\ad x)(\lf)\oplus(\ad x)(\qq)=(\ad x_n)(\lf)\oplus\qq$. Intersect with $\hh_1$ yields $\hh_0\cdot x=\lf_0\cdot x_n\oplus\qq_{1}$. In particular, $\codim_{\hh_1} H_0\cdot x=\codim_{\lf_1} L_0\cdot x_n$.

(iv) From the computations in (iii), we also get that $\hh\cdot x$ has trivial intersection with $\zz_{\ll}$. Since $\hh_{0}\cdot x+(\zz_{\ll})_1$ is the tangeant space of $J(\lf,\OO)$ at $x$, we get $\dim J(\ll,\OO)=\dim H_0\cdot x+\dim (\zz_{\ll})_1$. Since the elements of $J(\ll,\OO)$  all share the same orbit dimension, we have $\mod(H_0,J(\ll,\OO))=\dim J(\ll,\OO)-\dim H_0\cdot x$. The result follows.
%
\end{proof}

From Proposition \ref{dec_classes}, the decomposition classes satisfy the requirements of the covers used in Section \ref{gener}. However, note that two different data $(\lf_i,\OO_i)$, $i=1,2$ can give rise to the same decomposition class.
In the sequel, we fix a Cartan subspace $\cc\subset \hh_1$, that is a maximal subspace of commuting semisimple elements. The number $\dim \cc$ is called the \emph{rank} of $(H_0,\hh_1)$.

\begin{definition}  \label{restriction}
We say that a property $\mathcal P$ is \emph{hereditary} if, whenever $\mathcal P$ holds on a $\theta$-representation $(H_0,\hh_1)$, it holds on the restricted $\theta$-representation $((H_0^s)^{\circ}, \hh_1^t)$ for any semisimple element $t\in \hh_1$.  Here, $(H_0^s)^{\circ}$ is the connected subgroup of $H_0$ whose Lie algebra is $\gg_0^s$.
\end{definition}

Next, we recall a few classical facts.
\begin{lemma}\label{Vin_prelim} 
Let $(H_0,\hh_1)$ be a $\theta$-representation with Cartan subspace $\cc\subset \hh_1$.
\begin{enumerate}[label=(\roman*)]
\item For an element in general position $s\in \cc$, we have $\hh^s=\hh^\cc$ and $(\zz(\hh^s))_1=\cc$.
\item $(H_0,\hh_1)$ is locally free if and only if $\dim \hh_1=\dim \hh_0+\dim \cc$.
\item$(H_0,\hh_1)$ is stable if and only if there exists $s\in \cc$ such that $\hh_1^s=\zz(\hh^s)_1=\cc$.
\item Stability and local freeness are both hereditary properties.
\end{enumerate}
\end{lemma}
\begin{proof}
(i) follows from \cite[\S3.2]{Vin76}.\\
Since $\theta$-representations are visible, it follows from Remark \ref{rk_F0F1} that $\mod(H_0,\hh_1)= \dim \hh_1-m_0$ where $m_0$ is the maximal orbit dimension in $\hh_1$.
By (i), Proposition \ref{dec_classes} (iv) and \eqref{modTVJ}, the maximal modality of a decomposition class is $\mod(H_0,\hh_1)=\dim \cc$.
Assertion (ii) follows. \\
From Proposition \ref{dec_classes} (iii), an element $x=x_s+x_n$ can have maximal orbit dimension in $\hh_1$ only if $L_0 \cdot x_n$ is of maximal dimension in $\lf_1$ with $\lf=\hh^{x_s}$. This is possible with $x_n=0$ only if $\lf'_1=\{0\}$ where $\lf'=[\lf,\lf]$. Since $\lf_1$ contains a Cartan subspace, this implies that $\lf_1$ is a Cartan subspace. On the other hand, assume that there exists $s\in \cc$ such that $\hh_1^s=\zz(\hh^s)_1=\cc$. Then $J(\hh^{s},\{0\})$ is a decomposition class of dimension $\dim H_0\cdot s+\mod(H_0,J(\hh^{s},\{0\}))=\dim \hh_1-\dim \hh^s_1+0+\dim \zz(\hh^s)_1=\dim \hh_1$ by Proposition~\ref{dec_classes}. Hence $J(\hh^{s},\{0\})$ is the open class in $\hh_1$. Since it is made of semisimple elements, the representation is stable.\\
(iv) follows from \cite[Lemma~6.3]{BLLT}.
\end{proof}

Recall from \cite[Proposition~4.3]{BLLT} that $\muz$ can be irreducible only if $(H_0,\hh_1)$ is stable.

\begin{theorem}\label{main_thm_theta}
Assume that $(H_0,\hh_1)$ is a stable and locally free $\theta$-representation with $\hh$ simple, then 
\begin{enumerate}[label=\roman*)]
\item $\muz$ is an irreducible and reduced complete intersection
\item $\muz$ is normal if and only if $(H_0,\hh_1)$ does not belong to Table~\ref{ex_bad}
\end{enumerate}
\end{theorem}
\begin{corollary}
Under the hypothesis of the previous theorem and if $\hh$ is of classical type, then $\muz$ is normal 
\end{corollary}

\begin{table}
\begin{minipage}{.70\linewidth}
\hspace{-1cm}\begin{tabular}{|c|c | c|c |c|}
\hline
Type& Kac diagram &m& $n$ & $\wfr$ \\
\hline
$E_6^{(1)}$ &\scalebox{0.7}{$\begin{array}{c c c c c}1 & 1 &0 & 1& 1\\ &&1 &&\\&&1&&\end{array}$}& 9&
\scalebox{1}{$\begin{array}{l} x_{\alpha_2+\alpha_4}\\+x_{\alpha_3+\alpha_4}\\+x_{\alpha_5+\alpha_4}\end{array}$}&
$\langle x_{\alpha_4}\rangle$\\
\hline
$E_7^{(1)}$ & \scalebox{0.7}{$\begin{array}{c c c c c c c}1 & 1 & 1&0 & 1& 1&1\\ &&&1 &&&\end{array}$}&14&idem &idem\\
\hline
\multirow{3}{*}{$E_8^{(1)}$}
&\scalebox{0.7}{$\begin{array}{c c c c c c c c} 1&1&0&1&1&1 &1&1\\ &&1&&&&&\end{array}$}&24&idem &idem\\
\cline{2-5}
& \scalebox{0.7}{$\begin{array}{c c c c c c c c} 1&1&0&1&0&1 &1&1\\ &&1&&&&&\end{array}$}&20&idem &idem\\
\cline{2-5}
&\scalebox{0.7}{$\begin{array}{c c c c c c c c} 1&0&1&0&1&0 &1&1\\ &&0&&&&&\end{array}$}&15&
\scalebox{0.9}{$\begin{array}{l}x_{\alpha_2+\alpha_3+\alpha_4}\\
+x_{\alpha_2+\alpha_4+\alpha_5}\\+x_{\alpha_5+\alpha_6+\alpha_7}\\+x_{\alpha_7+\alpha_8}\end{array}$} &$\langle x_{\alpha_2}, x_{\alpha_7}\rangle$\\
\hline
\end{tabular}
\caption{Exceptional non-normal}\label{ex_bad}
\end{minipage}\begin{minipage}{.40\linewidth}
\begin{tabular}{|c| c|}
\hline
Type  & Kac diagram \\
\hline
$G_2^{(1)}$ &\scalebox{1.5}{\tiny
\rule[-3ex]{0pt}{5ex}
\begin{picture}(40, 0)
\put(11,2){\circle{4}}
\put(9,-5){0}
\put(22,0){\vector(-1,0){10}}
\put(21,2){\line(-1,0){8}}
\put(22,4){\vector(-1,0){10}}
\put(23,2){\circle{4}}
\put(21,-5){1}
\put(25,2){\line(1,0){4}}
\put(31,2){\circle{4}}
\put(29,-5){1}
\end{picture}
}\\
\hline
$F_4^{(1)}$ & \scalebox{1.5}{\tiny
\rule[-3ex]{0pt}{5ex}
\begin{picture}(40, 0)
\put(1,2){\circle{4}}
\put(-1,-5){1}
\put(3,2){\line(1,0){4}}
\put(9,2){\circle{4}}
\put(7,-5){0}
\put(17.5,1){\vector(-1,0){7}}
\put(17.5,3){\vector(-1,0){7}}
\put(19,2){\circle{4}}
\put(17,-5){1}
\put(21,2){\line(1,0){4}}
\put(27,2){\circle{4}}
\put(25,-5){1}
\put(29,2){\line(1,0){4}}
\put(35,2){\circle{4}}
\put(33,-5){1}
\end{picture}}\\
\hline
$E_6^{(2)}$& 
\scalebox{1.5}{\tiny
\rule[-3ex]{0pt}{5ex}
\begin{picture}(40, 0)
\put(1,2){\circle{4}}
\put(-1,-5){1}
\put(3,2){\line(1,0){4}}
\put(9,2){\circle{4}}
\put(7,-5){1}
\put(10.5,1){\vector(1,0){7}}
\put(10.5,3){\vector(1,0){7}}
\put(19,2){\circle{4}}
\put(17,-5){0}
\put(21,2){\line(1,0){4}}
\put(27,2){\circle{4}}
\put(25,-5){1}
\put(29,2){\line(1,0){4}}
\put(35,2){\circle{4}}
\put(33,-5){1}
\end{picture}}\\
\hline
\end{tabular}
\caption{ Exceptional normal}\label{ex_good}
\end{minipage}
\end{table}

From \cite[Theorem~3.2]{Pan94} (or, alternatively, Remark \ref{rk_F0F1} and Corollary~ \ref{cor_locfree}), (i) holds. 
The remaining of the section is devoted to the proof of Theorem~\ref{main_thm_theta}~(ii).  For this, we first rephrase condition \eqref{carac_smooth_codim1} in Lemma \ref{lemmaP}.  Then we list all the $\theta$-representations of rank one in Proposition \ref{init}. In Lemma \ref{P_or_not_P}, we check which of them satisfies condition \eqref{propP}. Finally, we check that the non-normal cases of rank one do not yield any non-normal example of greater rank in Lemma \ref{nonproper}.

Recall that all Cartan subspaces of $\hh_1$ are conjugate.
From Lemma \ref{Vin_prelim} (i) (iii) and Proposition \ref{dec_classes} (iv), the decomposition class $J(\hh^{\cc},\{0\})$ is the only one with maximal modality $\mod(H_0,\hh_1)=\dim \cc$. Moreover, the decomposition classes with $\mod(H_0,J(\lf,\OO))=\dim \cc-1$ are the ones such that $\dim (\zz_{\lf})_1=\dim \cc-1$ and $\OO$ arbitrary. The condition on $\lf$ means that $(L_0,\lf'_1)$ is a $\theta$-representation of rank one, where $\lf'$ is the derived subalgebra of $\lf$.

Recall that, for $x\in \hh_1$ with $\lf=\hh^{x_s}$, we have $\hh_0^{x,y}\subset\hh^{x}\subset \lf$ and that $\lf_0=\lf'_0$ by Lemma~\ref{Vin_prelim}~(iv). Hence, using Corollary \ref{cor_locfree} (v), we have:
\begin{lemma}\label{lemmaP}
Under hypothesis of Theorem \ref{main_thm_theta}, $\muz$ is normal if and only if
\begin{equation}\label{propP}
\textrm{$\forall n\in \lf'_1$ nilpotent, $\exists n'\in (\lf')_{-1}^n$ such that $(\lf')_0^{(n,n')}=\{0\}$.} \tag{$\mathcal P$} 
\end{equation}
for any Levi subalgebra $\lf\subset \hh$ arising from $\hh_1$ such that $(L_0,\lf'_1)$ is of rank one:
\end{lemma}

In other words, $\muz$ is normal if and only if the same holds for the semisimple part of any Levi subalgebra arising from $\hh_1$ whose associated $\theta$-representation is of rank one. The main difference with the symmetric Lie algebra case studied in \cite{Pan94} is that we have many $\theta$-representations of rank one. We now aim to classify them. For this, we recall a few features of the so-called \emph{Kac diagrams}. The reader is referred to \cite[\S 3]{Kac80} or \cite[\S 8]{Vin76} for a detailed treatment. 
\begin{itemize}
\item A Kac diagram is an affine (possibly twisted) Dynkin diagram $X_n^{(\ell)}$ equipped with an integer label for each node.
Up to isomorphisms of representations, it is enough to consider labels in $\{0,1,2\}$. If we focus on representations of positive rank, one can restricts to labels in $\{0,1\}$.
\item If $X_{n}(\ell)$ with a labelling $(v_0, \dots, v_n)$ form a Kac diagram, then it gives rise to a $\Zm$-grading on a reductive Lie algebra $\hh$, where $m=\ell \sum v_i$.
\item The Lie algebra $\hh_0$ is a reductive algebra whose  semisimple part is the Lie algebra corresponding to the subdiagram of $0$-labelled nodes. The rank of $\hh_0$ is that of the finite Dynkin diagram giving rise to $X_n^{(\ell)}$.
\end{itemize}

For instance, for each affine Dynkin diagram $X_n^{(\ell)}$ appearing in the classification of Kac \cite[\S3, Table 1]{Kac80}, we can label each node of $X_n^{(\ell)}$ with $1$. We denote such a Kac diagram by $\Xnko$. They give rise to the only $\theta$-representations of positive rank $(H_0,\hh_1)$ where $H_0$ is a torus.
\begin{lemma}\label{111}
Let $(H_0, \hh_1)$ be a $\theta$-representation corresponding to a Kac diagram of the form $\Xnko$.
Then $(H_0,\hh_1)$ is stable locally free of rank $1$, and $\muz$ is a normal variety.
\end{lemma}

\begin{proof}
In this case, $\hh_0$ is just the Cartan subalgebra of $\hh$ (associated with the Dynkin diagram without the extended node $\alpha_{n+1}$) and $\hh_1$ is the sum of root spaces $\hh_{\alpha_i}$, $i\in [\![1,n+1]\!]$. Local freeness is clear since $\alpha_1,\dots \alpha_n$ is a basis of $\hh_0^*$. By Lemma \ref{Vin_prelim} (ii), the rank of $(H_0,\hh_1)$ is one. Since $\alpha_{n+1}$ is the opposite of the highest root, there is a single relation linking the $\alpha_i$: $\alpha_{n+1}=\sum_{i\in [\![1,n]\!]}a_i\alpha_i$ with $a_i<0$. Then $(H_0, \hh_1)$ is stable by Proposition \ref{carac_stable},  $\muz$ is irreducible by Corollary \ref{st_irr}  so $\muz$ is normal by Theorem \ref{irrcomp}.
%
\end{proof}

\begin{proposition} \label{init}
Let $(H_0,\hh_1)$ be a stable and locally free $\theta$-representation of rank $1$ obtained from a simple Lie algebra $\hh$.
Then $(H_0,\hh_1)$  corresponds to a Kac diagram of the form $\Xnko$ or which appears in Table \ref{ex_bad} or Table \ref{ex_good}  
\end{proposition}

\begin{proof}
When the Lie algebra $\hh$ is \underline{exceptional} we use the software \cite{GAP} and its package \cite{SLA}. We proceed as follows.
 
Given $\hh$ of exceptional type (respectively of type $D_4$), the first step consists in listing the Kac-diagrams relative to $\hh$ (resp. with underlying Dynkin diagram $D_4^{(3)}$) with labels in $\{0,1\}$. Then we select the corresponding gradings for which $\dim \hh_1=\dim \hh_0+1$. There are few of them. For each selected grading, we choose a random element $x\in \hh_1$ and compute its semisimple part $s$. Since the rank of $(H_0,\hh_1)$ is at least $1$, the element $s$ should be non-zero (otherwise choose another $x$). Then, $(H_0,\hh_1)$ is stable locally free of rank $1$ if and only if $\dim \hh_{1}^s=1$ (Lemma~\ref{Vin_prelim}). In addition to the $\Xnko$, the only other Kac diagrams arising in this way are those of Tables \ref{ex_bad} and \ref{ex_good}. \\

Assume now that $\hh \subset\ggl(V)$ is \underline{classical}. Our goal is to show that the only stable and locally free $\theta$-representatons of rank $1$ are those arising from the $\Xnko$. The strategy is similar to the exceptional case, except that we rather rely on the description of classical $\theta$-representations of Vinberg \cite[\S7]{Vin76}.  We refer to this paper and we use the notation therein.  In particular, there are four subcases to consider: 
\begin{description}
\item[\normalfont{\emph{First case}}] $\theta$ is an inner automorphism of $\hh=\liesl(V)$ (type $A$);
\item[\normalfont{\emph{Second case}}] $\hh$ is of type $B$ or $D$; 
\item[\normalfont{\emph{Third case}}] $\hh$ is of type $C$; 
\item[\normalfont{\emph{Fourth case}}] $\theta$ outer and $\hh$ of type $A$.
\end{description}
In the first three cases (respectively in the fourth case) $\theta$ (respectively $\theta^{2}$) is an automorphism of order $m_0$ given by $g\mapsto aga^{-1}$ for some $a \in \GL(V)$ satisfying $a^{m_{0}}\in \{\pm Id\}$. For $\theta$ of order $1$ or $2$, the only locally free case of rank $1$ is $A^{(1)}_1(\ubar 1)$ so we may assume that $m_0> 1$. Then $V=\bigoplus_{\lambda\in S}V(\lambda)$ where $V(\lambda)=\{x\in V| a.x=\lambda x\}$ and $S=\{\lambda| \lambda^{m_0}=+1\}$ or $S=\{\lambda | \lambda^{m_0}=-1\}$. That is $S=\{\lambda_0\omega_0^j|j\in \Zmo\}$ for $\lambda_0=1$ or $exp(\frac{2 i \pi}{2m_0})$ and $\omega_0:=exp(\frac{2i\pi}{m_0})$. 
For $j\in \Zmo$, we set $\lambda_j:=\lambda_0\omega_0^j$ and $k_j:=\dim V(\lambda_j)$. The elements $x\in\hh_k$ satisfy $x.V(\lambda_j)\subset V(\lambda_{j+k})$ for any $j\in \Zmo$. 

We first look at the first case. From \cite{Vin76}, the rank of $(H_0,\hh_1)$ is $\min_{j\in \Zmo} k_j$. Since $\hh_0\cong\liesl(V)\cap \prod_{j\in \Zmo}\ggl(V_j))$ and $\hh_1\cong \prod_{j\in \Zmo}Hom(V_j,V_{j+1})$, we have \[\dim \hh_1-\dim \hh_0=
\left (\sum_{j\in \Zmo} k_j k_{j+1}\right)-\left(\sum_{j\in \Zmo} k_j^2-1\right)=1-\frac12 \sum_{j\in \Zmo} \left(k_j-k_{j+1}\right)^2.\] 

From Lemma~\ref{Vin_prelim}~(ii), we want that this number is $1$. The only possibility is that $k_j=k_{j'}$ for any $j,j'\in \Zmo$ and the condition on the rank forces this common value to be $1$. In this particular case, $H_0$ is a torus so this must coincide with a grading given by a Kac diagram of the form $A^{(1)}_n(\ubar 1)$ (inner automorphism).

From now on we focus on the last three cases. In each of these cases, a nondegenerate bilinear form comes into play. This  form induces a duality between $V(\lambda)$ and $V(\overline \lambda)$ and we have $k_j=k_{\bar j}$ where $\bar j$ denotes the element of $\Zmo$ satisfying $\lambda_{\bar j}=\overline{\lambda_j}$. Depending on whether each elements $1$ and $-1$ belong to $S$ or not, the corresponding $\theta$-representations may have different behavior. For $\nu\in \{\pm1\}$, we set $\eta_{\nu}:=\alter{1 & \textrm{ if } \nu\in S\\ -1 & \textrm{ if } \nu \notin S}$. 
Set also $\epsilon_1:=\alter{1& \textrm{ in the 2nd and 4th case}\\-1 &\textrm{ in the 3rd case}}$ and $\epsilon_{-1}:=\alter{1& \textrm{ in the 2nd case}\\-1 &\textrm{ in the 3rd and 4th case}}$. 
We then have a uniform formula
\begin{eqnarray*}\dim \hh_1-\dim \hh_0&=&
\frac12\left[\sum_j k_j k_{j+1}-\sum_j k_j^2+(\epsilon_1\eta_1k_0+\epsilon_{-1}\eta_{-1}k_{m'})\right]\\
&=& \frac 14\left[-\sum_{j}(k_j-k_{j+1})^2+2\epsilon_1\eta_1k_0+2\epsilon_{-1}\eta_{-1}k_{m'}\right]
\end{eqnarray*}
In a nutshell, $\hh_0$ is a product of the $\ggl(V_j)\cong\ggl(V_{\bar j})$ for $\lambda_j\notin\{\pm1\}$ and possible copies of $\mathfrak{so}(V_j)$ or $\mathfrak{sp}(V_j)$ for $\lambda_j\in\{\pm1\}$. A similar combinatoric holds for $\hh_1$. In the fourth case, the assumption $m_0>1$ ensures  that $\K Id_V$ lies in a $k$-th part of the grading in $\ggl(V)$ satisfying $k\notin\{0,1\}$.

We choose $m'\in \Zmo$ such that $Re(\lambda_{m'})$ is minimal. By convention, the argument of a complex number is taken in $[0,2\pi[$. Choosing $j_{min}\in \Zmo$ such that $k_{j}$ is minimal and $Im(\lambda_{j})\geqslant 0$, we define
\[I=\left\{j\, \left| arg(\lambda_j)<arg(\lambda_{j_{min}}) \textrm{ or }arg(\lambda_{m'}) \leqslant arg(\lambda_j)<arg(\lambda_{\overline{j_{min}}}) \right.\right\},\] 
\[I=\left\{j\, \left|arg(\lambda_{j_{min}}) \leqslant arg(\lambda_j)<arg(\lambda_{m'}) \textrm{ or } arg(\lambda_{\overline {j_{min}}}) \leqslant arg(\lambda_j) \right.\right\}.\] 

Using the equality $n^2=n(n-1)+n$, we can write
\begin{eqnarray*}\dim \hh_1-\dim \hh_0\!\!\!&=&\!\!\!
\frac14 \left[-\sum_{j\in I} (k_j-k_{j+1})(k_j-k_{j+1}-1)
 -\sum_{j\in J} (k_{j+1}-k_{j})(k_{j+1}-k_{j}-1)\right.\\ 
 & &\left.+\sum_{j\in I} (k_{j+1}-k_j)+\sum_{j\in J} (k_{j}-k_{j+1})+2\epsilon_1\eta_1k_0+2\epsilon_{-1}\eta_{-1}k_{m'}\right]\end{eqnarray*}
The first two terms are non-positive and the rest of the expression is equal to $(k_{j_{min}}-k_0)+(k_{\overline{j_{min}}}-k_{m'})+(k_{j_{min}}-k_{m'})+(k_{\overline {j_{min}}}-k_0)+2\epsilon_1\eta_1k_0+2\epsilon_{-1}\eta_{-1}k_{m'}=2\left[(\epsilon_1\eta_1-1)k_0+k_{j_{min}}\right]
+2\left[(\epsilon_{-1}\eta_{-1}-1)k_{m'}+k_{j_{min}}\right]$. Each of the terms in brackets is either equal to $k_{j_{min}}$ if $\epsilon_{\nu}\eta_{\nu}=1$, or is not greater than $-k_{j_{min}}$ if $\epsilon_{\nu}\eta_{\nu}=-1$. In particular, $\dim \hh_1-\dim \hh_0$ can be positive only if $\epsilon_{\nu}\eta_{\nu}=1$ for $\nu\in \{\pm1\}$. Hence, it is of type I in the terminology of \cite[\S7.2]{Vin76}, and its rank is $k_{j_{min}}$.
Then, it follows from our formulas and Lemma \ref{Vin_prelim} that the only locally free cases of rank $1$ are the following ones 
\[
\left(\begin{array}{c}k_{j_{min}}=1,\\ \forall j\in I, k_j-k_{j+1}\in \{0,1\},\\ \forall j\in J, k_{j+1}-k_{j}\in \{0,1\}\end{array}\right) \textrm{ and }\alter{\pm1\in S&\textrm{ in the 2nd case}\\\pm1\notin S&\textrm{ in the 3rd case}\\1\in S, -1\notin S&\textrm{ in the 4th case}} 
.\]

In the last three cases, the action is not stable in general. Given subspaces $U(\lambda)\subset V(\lambda)$ of dimension $1$, we can construct a non-zero semisimple element $C\in \hh_1$ such that $C(U(\lambda))=U(\omega_0\lambda)$. Then $V':=Ker(C)$ decomposes as $V':=\bigoplus_{\lambda} V'(\lambda)$; the decomposition in two summands $V=(\bigoplus_{\lambda} U(\lambda))\oplus V'$ is orthogonal and the derived subalgebra of $\ggl(V)^C$ can be identified with $\liesl(V')$. If there exists $j\in \Zmo$ such that $k_j,k_{j+1}\geqslant 2$, then we can construct in each case a non-zero element $N$ in $\hh_1^C\cap \liesl(V')$ with $N(V'({\lambda_j}))\subset V'(\lambda_{j+1})$, $N(V'(\lambda_{\overline{j+1}}))\subset V'(\lambda_{\overline{j}})$ and $N(U(\lambda))=0=N(V'(\lambda_k))$ for any $\lambda\in S$ and any $k\notin \{j,\overline{j+1}\}$. This contradicts the stability hypothesis since $N\in \hh^C\setminus \K C$, see Lemma \ref{Vin_prelim}~(iii).
Thus, the only stable gradings to consider are those with $\alter{k_j=1& \textrm{for $j$ such that $\lambda_j \notin\{\pm1\}$}\\ k_j\in\{1,2\} & \textrm{for $j$ such that $\lambda_j \in\{\pm1\}$}\\ }$. The different possibilities are listed in Table \ref{listclassical} (up to multiplication of $a$ by $-1$ in type $B$, which gives rise to the same grading).
Once again, $H_0$ is a torus in each case so the corresponding grading must be obtained from a Kac diagram of the form $\Xnko$. These Kac diagrams are identified in the last column of Table \ref{listclassical}.
\end{proof}

\begin{table}
\hspace{-1cm}
\begin{tabular}{|c| c| c| c| c|}
\hline
Case & \begin{tabular}{c}Parity \\of $m_0$\end{tabular} & \begin{tabular}{c}Parity of\\ $dim(V)$\end{tabular} & Condition on the $k_j$& $\Xnko$\\
\hline
First case & / &/  & $\forall j, \, k_j=1$& $A_n^{(1)}(\ubar1)$\\
\hline
\multirow{3}{*}{Second case} & \multirow{3}{*}{even}& odd  & $k_j=\alter{2&\textrm{if $j=0$}\\1&\textrm{else}}$& $B_n^{(1)}(\ubar1)$\\
\cline{3-5}
&&\multirow{2}{*}{even}& $\forall j, \, k_j=1$ &$D_n^{(1)}(\ubar1)$\\
\cline{4-5}
&&&$k_j=\alter{2&\textrm{if $j\in\{0,\frac{m_0}2\}$}\\1&\textrm{else}}$&$D_n^{(2)}(\ubar1)$ \\
\hline
Third case & even & even &  $\forall j, \, k_j=1$ & $C_n^{(1)}(\ubar1)$\\
\hline
\multirow{2}{*}{Fourth case} &\multirow{2}{*}{odd}&even& $k_j=\alter{2&\textrm{if $j=0$}\\1&\textrm{else}}$ &$\begin{array}{c}A_{2p-1}^{(2)}(\ubar1)
\end{array}$\\
\cline{3-5}
&&odd&$\forall j, \, k_j=1$ & $\begin{array}{c}A_{2p}^{(2)}(\ubar1)
\end{array}$\\
\hline
\end{tabular}
\caption{Stable classical locally free grading of rank $1$}\label{listclassical}
\end{table}

\begin{lemma}\label{P_or_not_P}
If $(H_0, \hh_1)$ is a $\theta$-representation associated to a Kac diagram of Table \ref{ex_bad} (resp. Table \ref{ex_good}) then it does not satisfy (resp. it satisfies) condition \eqref{propP} of Lemma \ref{lemmaP}.
\end{lemma}
The proof of this lemma relies on computer-based computations. 
First, we list representatives of the nilpotent $H_0$-orbits on $\hh_1$. For any such representative $n_i$, we choose a random element $n_i'\in \hh_{-1}^n$. Then we compute $\hh_0^{(n_i,n_i')}$. It turns out to be $\{0\}$ for each $i$ in cases of Table \ref{ex_good} so \eqref{propP} holds. In cases of Table \ref{ex_bad}, we compute $\wfr_i:=\{y\in \hh_0^{n_i}\,|\,[y,\hh_{-1}^{n_i}]=\{0\}\}=\bigcap_{z\in \hh_{-1}^{n_i}} \hh_0^{(n_i,z)}$. In each case, it turns out that there exists $i$ such that $\wfr_i\neq\{0\}$. Hence \eqref{propP} must fail in these cases. An example of such $n_i, \wfr_i$ is given in the table, using the numbering of the roots as in Bourbaki and where $x_{\alpha}$ denotes a root vector for the root $\alpha$.

\begin{lemma}\label{nonproper}
Let $(H_0,\hh_1)$ be a locally free stable $\theta$-representation with $\hh$ simple. Then there is no proper Levi $\lf$ arising from $\hh_1$ such that $(L_0, \lf_1')$ corresponds to a Kac diagram of Table \ref{ex_bad}.   
\end{lemma}

\begin{proof}
Assume that such a proper Levi exists. Then $\hh$ is of type $E_7$ or $E_8$ and the Cartan subspace of $\hh_1$ is of dimension $2$ or $3$. 
Moreover the order of the automorphism $\theta$  giving rise to the graduation is a multiple of $9$ or $14$. 
We can now consider all the possible Kac diagrams of type $E_7$ or $E_8$ such that $\dim \gg_1-\dim \gg_0 \geq 2$. 
After computation, it turns out that the resulting $\Zm$-gradings have order which are not multiple of $9$ or $14$, which is a contradiction.
\end{proof}

\begin{remark} \label{norm_quot}
Conjecture \ref{conjA} in the case of a locally free stable $\theta$-representation follows from \cite[Theorem~1.2]{BLLT}. We briefly explain here how to recover it using the results of this paper. First, note that  the $5$ exceptional cases of Table \ref{ex_bad} are of rank one. Then it follows from \cite[Proposition 3.4 \& 5.1]{BLLT} that the natural map $\cc\oplus\cv/W\rightarrow V\oplus V^*\SPR G$ is a dominant closed immersion, hence an isomorphism. In the other cases, Theorem \ref{main_thm_theta} ensures that $\muz$ is normal. So the same holds for the symplectic reduction and the result follows from \cite[Proposition~5.1]{BLLT} and Remark~\ref{rk_conjA}.\\
 Note that the ground results on polar representations of \cite{DK85} are known for $\theta$-representations over any algebraically closed field of characteristic zero. Thus, we don't need the assumption $\K=\C$ used in \cite{BLLT}. 
\end{remark}

\end{document}